\documentclass[12pt]{article}
\usepackage{amsfonts}
\usepackage{bm}
\usepackage{amsmath, amsthm, amsfonts, amssymb,  color}
\usepackage{mathrsfs}
\newtheorem{thm}{Theorem}[section]

\newtheorem{lem}[thm]{Lemma}

\newtheorem{exa}[thm]{Example}
\theoremstyle{definition}
\newtheorem{defn}{Definition}[section]
\newcommand{\scr}[1]{\mathscr #1}
\definecolor{wco}{rgb}{0.5,0.2,0.3}

\numberwithin{equation}{section} \theoremstyle{remark}
\newtheorem{rem}{Remark}[section]

\newcommand{\ua}{\uparrow}

\renewcommand{\hat}{\widehat}

\title{{\bf Approximation of Invariant Measures for   Regime-Switching Diffusions}
}
\author{Jianhai Bao,\thanks{Department of Mathematics, Central South University,
Changsha, Hunan, 410083,   China, jianhaibao13@gmail.com} \and
Jinghai Shao\thanks{School of Mathematical Science, Beijing Normal
University,  Beijing, 100875, China, shaojh@bnu.edu.cn} \and
Chenggui Yuan\thanks{Department of Mathematics, Swansea University,
Singleton Park, SA2 8PP, UK, C.Yuan@swansea.ac.uk}}
\begin{document}
\def\R{\mathbb R}  \def\ff{\frac} \def\ss{\sqrt} \def\B{\mathbf
B}
\def\N{\mathbb N} \def\kk{\kappa} \def\m{{\bf m}}
\def\dd{\delta} \def\DD{\Dd} \def\vv{\varepsilon} \def\rr{\rho}
\def\<{\langle} \def\>{\rangle} \def\GG{\Gamma}
  \def\nn{\nabla} \def\pp{\partial} \def\EE{\scr E}
\def\d{\text{\rm{d}}} \def\bb{\beta} \def\aa{\alpha} \def\D{\scr D}
  \def\si{\sigma} \def\ess{\text{\rm{ess}}}
\def\beg{\begin} \def\beq{\begin{equation}}  \def\F{\scr F}
\def\Ric{\text{\rm{Ric}}} \def\Hess{\text{\rm{Hess}}}
\def\e{\text{\rm{e}}} \def\ua{\underline a} \def\OO{\Omega}  \def\oo{\omega}
 \def\tt{\tilde} \def\Ric{\text{\rm{Ric}}}
\def\cut{\text{\rm{cut}}} \def\P{\mathbb P} \def\ifn{I_n(f^{\bigotimes n})}
\def\C{\scr C}      \def\aaa{\mathbf{r}}     \def\r{r}
\def\gap{\text{\rm{gap}}} \def\prr{\pi_{{\bf m},\varrho}}  \def\r{\mathbf r}
\def\Z{\mathbb Z} \def\vrr{\varrho} \def\l{\lambda}
\def\L{\scr L}\def\Tt{\tt} \def\TT{\tt}\def\II{\mathbb I}
\def\i{{\rm in}}\def\Sect{{\rm Sect}}\def\E{\mathbb E} \def\H{\mathbb H}
\def\M{\scr M}\def\Q{\mathbb Q} \def\texto{\text{o}} \def\LL{\Lambda}
\def\Rank{{\rm Rank}} \def\B{\scr B} \def\i{{\rm i}} \def\HR{\hat{\R}^d}
\def\to{\rightarrow}\def\l{\ell}\def\lf{\lfloor}\def\rf{\rfloor}
\def\8{\infty}\def\ee{\epsilon} \def\Y{\mathbb{Y}} \def\lf{\lfloor}
\def\rf{\rfloor}\def\3{\triangle} \def\O{\mathcal {O}}
\def\SS{\mathbb{S}}\def\ta{\theta}\def\h{\hat}

\def\la{\langle}\def\ra{\rangle}

\def\trace{\text{\rm{trace}}}
\renewcommand{\bar}{\overline}
\def\Y{\mathbb{Y}}
\renewcommand{\tilde}{\widetilde}
%\date{}

\maketitle

\begin{abstract}
In this paper, we are concerned with long-time behavior of
Euler-Maruyama schemes associated with a range of regime-switching
diffusion processes. The key contributions of this paper lie  in
that  existence and uniqueness of numerical invariant measures are
addressed (i) for regime-switching diffusion processes with finite
state spaces by the Perron-Frobenius theorem  if the ``averaging
condition" holds, and, for the case of reversible Markov chain, via
the principal eigenvalue approach provided that the principal
eigenvalue is positive; (ii) for regime-switching diffusion
processes with countable state spaces by means of a finite partition
method  and an M-Matrix theory. We also
 reveal that numerical invariant measures converge  in
the Wasserstein metric to the underlying ones. Several
examples are constructed to demonstrate our theory.

 \end{abstract}

\section{Introduction}
For a regime-switching diffusion process (RSDP), we mean a diffusion
process in a random environment characterized by a Markov chain. The
state vector of an RSDP is a pair $(X_t,\LL_t)$. Here
$\{X_t\}_{t\ge0}$ satisfies a stochastic differential equation (SDE)
\begin{equation}\label{eq1}
\d X_t=b(X_t,\LL_t)\d t+\si(X_t,\LL_t)\d
W_t,~t>0,~X_0=x\in\R^n,\LL_0=i\in\mathbb{S},
\end{equation}
and $\{\LL_t\}_{t\ge0}$ denotes
% stands for a random switching device
%formulated as
  a continuous-time Markov chain with the
state space $\mathbb{S}:=\{1,2\cdots,N\}$, $1\le N\le\8,$  and the
transition rules
%of $\{\LL_t\}_{t\ge0}$
specified by
\begin{equation}\label{love}
\P(\LL_{t+\3}=j|\LL_t=i)=
\begin{cases}
q_{ij}\3+o(\3),~~~~~~~~i\neq j,\\
1+q_{ii}\3+o(\3),~~~i=j.
\end{cases}
\end{equation}
Various quantities of \eqref{eq1} will be given
  in the next section.
RSDPs have considerable applications in e.g. control problems,
storage modeling, neutral activity, biology and mathematical finance
(see e.g. monographs \cite{YM06,YZ}). Pinsky and Scheutzow \cite{PS}
showed that the overall system $(X_t^x,\LL_t^i)$ need not to be
positive recurrence (resp. transience) even when each subsystem
 is positive recurrence (resp. transience), and Mao and
Yuan  revealed in \cite[Example 5.45, p.223]{YM06} that
$(X_t^x,\LL_t^i)$ is stable although some of the subsystems are not.
So, in some cases,  the dynamical behavior of RSDPs may be markedly
different from diffusion processes without regime switchings. So
far, the works on RSDPs have included ergodicity
\cite{BGM,CM13,S14,SX,Xi}, stability \cite{YM06,SX,XY,YZ},
recurrence and transience \cite{PP,PS,SJ,YZ}, invariant densities
\cite{B12,BHM}, hypoellipticity \cite{B12,BLM}, and so forth.

\smallskip

 Since solving RSDPs is still a
challenging task, numerical schemes and/or approximation techniques
have become one of the viable alternatives, where \cite{YM06,YZ} are
concerned with finite-time (strong or weak) convergence while
\cite{HMY,YM06} are devoted to long-time behavior of numerical
schemes. For more details on numerical analysis of diffusion
processes without regime switching, please refer to the monograph
\cite{KP}. Also, approximations of invariant measures for stochastic
dynamical systems have  attracted much attention, see e.g. Mattingly
et al. \cite{MST} via a Poisson equation, Talay \cite{T90} through
the Kolmogorov equation, and Br\'{e}hier \cite{B14} by means of the
Malliavin calculus. For the counterpart associated with
Euler-Maruyama (EM) algorithms
 with constant/decreasing stepsize of RSDPs, we refer to Mao
et al. \cite{MYY} and  Yuan and Mao \cite{YM05} adopting the M-matrix
theory, and Yin and Zhu \cite{YZ} utilizing the weak convergence
method, where   RSDPs
 therein enjoy finite state spaces. Moreover, sufficient conditions
imposed in \cite{MYY,YM05,YZ} to guarantee existence of numerical
invariant measures are irrelevant to stationary distributions of the
continuous-time Markov chains that can accommodate a set of possible
regimes.

\smallskip

Motivated by \cite{MYY,YM05,YZ}, in this paper  we are also
interested in numerical approximation of invariant measure for RSDP
\eqref{eq1} and \eqref{love}. In particular,  we are concerned with
the following questions:
\begin{enumerate}
\item[(i)] Under what conditions, will the discrete-time semigroup
generated by  EM scheme admit  a unique invariant measure?

\item[(ii)] Will the numerical invariant measure, if it exists,
converge in some metric to the underlying one?
\end{enumerate}
In this paper, we shall  answer the questions above one-by-one  in
several cases.

\smallskip
Throughout the paper,  we stipulate $N<\8$ in Section
\ref{sec2}-\ref{sec4}, and $N=\8$ in Section \ref{sec5}. The content
of this paper is arranged  as follows. In Section \ref{sec2}, by the
Perron-Fronenius theorem, we discuss existence and uniqueness of
invariant measure for semigroup generated by $(X_t^x,\LL_t^i)$ if
\eqref{eq1} is attractive ``in average" (see \eqref{eq8}). In what
follows, we call \eqref{eq8} an ``averaging condition". As Example
\ref{Exa} below shows, our established theory,   Theorem \ref{exa},
covers more interesting models in contrast to   existing results
(see e.g. \cite[Theorem 5.1]{YM03}). By following the idea of
argument for Theorem \ref{exa}, Section \ref{sec3} focus on
existence and uniqueness of numerical invariant measure for RSDP
\eqref{eq1} and \eqref{love} with additive noise and multiplicative
noise respectively. In addition, we also reveal that  numerical
invariant converges in the Wasserstein distance to the underlying
one. For more details, please refer to Theorem \ref{W_p}. We  point
out that the Markov chain considered in Section \ref{sec3} need not
to be reversible. However, for the reversible case, by the principal
eigenvalue approach (see e.g. Chen \cite{Ch}), existence and
uniqueness of numerical invariant measure can also be addressed if
the principal eigenvalue is positive.
% (see Theorem \ref{eigen}).
 This
is elaborated in Section \ref{sec4}. Note that the Markov chain in
Section \ref{sec3} and \ref{sec4} admits a finite state space. We
 proceed to the countable case in Section \ref{sec5}. For such case,
a finite partition method due to Shao \cite{S14} and an M-matrix
theory (see e.g. \cite[Theorem 2.10, p.68]{YM06}) are adopted to
study existence of numerical invariant measure. More precisely, by a
finite partition method,  EM scheme with a countable state space is
transformed into a new EM with a finite state space. Thus, the
discrete-time semigroup generated by the EM scheme with a countable
state space possesses an invariant measure provided that the one
generated by the new EM with a finite state space does. Moreover,
several examples (see Examples \ref{Exa}, \ref{Exa1}, \ref{ex4.3}
and Remark \ref{re3.1}) are constructed to demonstrate our  theory
established.

\smallskip

Throughout the paper, $c>0$ is a generic constant which is
independent of the time parameters and the stepsize, and may change
from occurrence to occurrence.

\section{Invariant Measure}\label{sec2}
To begin with, we introduce some notation. Let $(\OO, \F,\P)$ be a
probability space with a filtration $\{\F_t\}_{t\ge0}$ satisfying
the usual conditions (i.e. $\F_0$ contains all $\P$-null sets and
$\F_t=\F_{t+}:=\bigcap_{s>t}\F_s)$. Let $\{\LL_t\}_{t\ge0}$ be a
continuous-time Markov chain with the state space
$\mathbb{S}:=\{1,2\cdots,N\}$, $N<\8,$  and $\{W_t\}_{t\ge0}$ an
$m$-dimensional Brownian motion, independent of
$\{\Lambda_t\}_{t\ge0}$, defined on the probability space above. We
assume that the $Q$-matrix $Q:=(q_{ij})_{N\times N}$ is irreducible
and conservative. So the Markov chain $\{\LL_t\}_{t\ge0}$ has a
unique stationary distribution $\mu:=({\mu_1,\cdots,\mu_N})$ which
can be determined by solving linear equation $$\mu Q={\bf0}$$
subject to $\mu_1+\cdots+\mu_N=1 ~\mbox{ and
}~\mu_i>0,i\in\mathbb{S}.$ Here ${\bf0}$ is a zero vector. Let
$\mathcal {P}(\R^n\times\mathbb{S})$ stand for the family of all
probability measures on $\R^n\times\mathbb{S}$. For
$\xi=(\xi_1,\cdots,\xi_n)^*\in\R^n$,  $\xi\gg{\bf0}$ means each
component $\xi_i>0$, $i=1,\cdots,n,$ in which $A^*$ denotes the
transpose of a vector or matrix $A.$ For $a,b\in \R$, $a\wedge
b:=\min\{a,b\}$. For each $R>0$,  $B_R(0):=\{x\in\R^n:|x|\le R\}$,
 the ball of radius $R$ centered at $0$. Let $\|A\|$ be the
Hilbert-Schmidt norm of the matrix $A.$
$\mbox{diag}(a_1,\cdots,a_N)$ denotes the diagonal matrix whose
diagonal entries starting in the upper left corner are
$a_1,\cdots,a_N$.
%Moreover, in the sequel, we stipulate that $N<\8$ in Section
%\ref{sec2}-\ref{sec4}, and $N=\8$ in Section \ref{sec5}.

\smallskip
%In this paper, we focus on the RSDP  $(X_t,\LL_t)$ determined by
%\eqref{eq1} and \eqref{love}, where
We assume that, in \eqref{eq1},  $b:\R^n\times\mathbb{S}\mapsto\R^n$
and $\si:\R^n\times\mathbb{S}\mapsto\R^n\otimes\R^m$
 satisfy the local Lipschitz condition, i.e., for each
 $i\in\mathbb{S}$ and $
 R>0$,
 there exists an $L_R>0$ such that
\begin{equation}\label{*02}
|b(x,i)-b(y,i)|+\|\si(x,i)-\si(y,i)\|\le L_R|x-y|,~~~x,y\in B_R(0).
\end{equation}
%where  and $B_R(0)$ denotes the ball of radius $R$ centered at $0$.
% \eqref{eq5} implies
%linear growth condition:
%\begin{equation}\label{eq6}
%|b(x,i)|+\|\si(x,i)\|\le L_0+L|x|,~~~x\in\R^n,
%\end{equation}
%where $L_0:=\max_{i\in\mathbb{S}}\{|b(0,i)|+\|\si(0,i)\|\}.$ Then
%\eqref{eq1} admits a unique non-explosive strong solution
%$\{X_t^x\}_{t\ge0}$ with the staring point $x$ (see e.g.
%\cite[Theorem 3.8]{YM06}).

\smallskip
Additionally,  we assume that
\begin{enumerate}
\item[\textmd{({\bf H})}] For each  $i\in \mathbb{S}$ and $x,y\in\R^n$, there
exist $c_0>0$ and
 $\bb_i\in\R$ such that
\begin{equation}\label{r5}
2\<x,b(x,i)\> +\|\si(x,i)\|^2 \le c_0+\bb_i|x|^2,
\end{equation}
and
\begin{equation}\label{r6}
2\<x-y,b(x,i)-b(y,i)\> +\|\si(x,i)-\si(y,i)\|^2 \le\bb_i|x-y|^2.
\end{equation}
\end{enumerate}

\begin{rem}
{\rm In ({\bf H}), it is worth to pointing out that, for each $i\in
\mathbb{S}$, $\bb_i$ need not to be negative. On the other hand,
without loss of generality, we assume that, for each $i\in
\mathbb{S}$, \eqref{r5} and \eqref{r6} hold respectively with the
same $\bb_i\in\R$  to avoid complex computation.

 }
\end{rem}
% By \eqref{eq6} and ({\bf H}), for each $i\in \mathbb{S}$
%and any $\vv\in(0,1)$,    there exists $c_\vv>0$ such that
%\begin{equation}\label{eq7}
%\begin{split}
%2\<x,b(x,i)\>+\|\si(x,i)\|^2
%&=2\<x,b(x,i)-b(0,i)\>+\|\si(x,i)-\si(0,i)\|^2\\
%&\quad+2\<x,b(0,i)\>+2\<\si(x,i),\si(0,i)\>\\
%&\le c_\vv+(\bb_i+\vv)|x|^2,~~~x\in\R^n.
%\end{split}
%\end{equation}
 Under \eqref{*02} and
({\bf H}), \eqref{eq1} and \eqref{love} admit a unique non-explosive
solution $(X_t,\LL_t)$, see e.g. \cite[Theorem 3.17, p.93]{YM06}.
Throughout the paper, we write $(X_t^{x,i},\LL_t^i)$ in lieu of
$(X_t,\LL_t)$ to highlight the initial data $(x,i)$. Define a metric
$d(\cdot,\cdot)$ on $\R^n\times\mathbb{S}$ as below
\begin{equation*}
d((x,i),(y,j)):=|x-y|+d_0(i,j),
\end{equation*}
where $d_0(i,j)=0$ for $i=j$, otherwise  $d_0(i,j)=1$. Then
$(\R^n\times\mathbb{S},d(\cdot,\cdot),\B(\R^n\times\mathbb{S}))$ is
a complete separable metric space. For two given probability
measures $\mu$ and $\nu$ on $\R^n\times\mathbb{S}$, define
\begin{equation}\label{d3}
W_p(\mu,\nu):=\inf_{\pi\in \mathcal
{C}(\mu,\nu)}\int_{\R^n\times\mathbb{S}}\int_{
\R^n\times\mathbb{S}}d(x,y)^p\pi(\d x,\d y),~~p\in(0,1],
\end{equation}
where $\mathcal {C}(\mu,\nu)$ denotes the set of all couplings of
$\mu$ and $\nu.$ Let $P_t(x,i;\d y\times\{j\})$ be the transition
probability measure of the pair $(X_t^{x,i},\LL_t^i)$, which is a
time homogeneous Markov process (see e.g. \cite[Theorem 3.28,
p.105-106]{YM06}).
%, with the starting point $(x,i).$
%We call $X_t^x$
%the continuous and $\LL_t^i$ the discrete component of
%$(X_t^x,\LL_t^i)$.
Recall that $\pi\in\mathcal {P}(\R^n\times\mathbb{S})$ is called an
invariant measure of $(X_t^{x,i},\LL_t^i)$ if
\begin{equation*}
\pi(\Gamma\times\{i\})=\sum_{j=1}^N\int_{\R^n}P_t(x,j;\Gamma\times\{i\})\pi(\d
x\times\{j\}),~~~t\ge0,
\end{equation*}
holds for any Borel set $\GG\in\B(\R^n)$ and  $i\in\mathbb{S}.$
%\eqref{eq1} is said to admit an invariant measure if
%$(X_t^{x,i},\LL_t^i)$ does.
%We further need to introduce some additional notation.
For any
$p>0$, let
\begin{equation}\label{r4}
\mbox{diag}({\bf\bb}):=\mbox{diag}(\bb_1,\cdots,\bb_N),~~~~
Q_p:=Q+\ff{p}{2}\mbox{diag}(\bb),~~\eta_p:=-\max_{\gamma\in\mbox{spec}(
Q_p)}\mbox{Re}\gamma,
\end{equation}
where $Q$ is the $Q$-matrix of $\{\Lambda_t\}_{t\ge0},$ and
$\mbox{spec}(Q_p)$ denotes the spectrum of $Q_p.$

\smallskip

The lemma below plays a crucial role for existence   of
 an invariant measure of $(X_t^{x,i},\LL_t^i)$.

\begin{lem}\label{P-F}
{\rm (\cite[Proposition 4.2]{BGM}) Let $N<\8$ and assume that
\begin{equation}\label{eq8}
\sum_{i=1}^N\mu_i\bb_i<0.
\end{equation}
Then, one has
\begin{enumerate}
\item[(i)] $\eta_p>0$ if $\max_{i\in\mathbb{S}}\bb_i\le0;$

\item[(ii)] $\eta_p>0$ for  $p<k$, where $k\in(0,\min_{i\in\mathbb{S},\bb_i>0}\{-2q_{ii}/\bb_i\})$ with $\max_{i\in\mathbb{S}}\bb_i>0.$
\end{enumerate}
}
\end{lem}

\begin{rem}
{\rm The RSDP \eqref{eq1} and \eqref{love} is said to be attractive
``in average'' if \eqref{eq8} holds (see e.g. Bardet et al.
\cite{BGM}).
%In what
%follows, we call \eqref{eq8} an ``averaging condition'' for
%presentation convenience.

}
\end{rem}

Our first main result in this paper is stated as below.
\begin{thm}\label{exa}
{\rm Let $N<\8$ and assume that \eqref{*02}, ({\bf H}) and
\eqref{eq8} hold. Then $(X_t^{x,i},\LL_t^i)$ admits a unique
invariant measure $\pi\in\mathcal {P}(\R^n\times\mathbb{S})$. }
\end{thm}

\begin{proof}
Let $Q_{p,t}:=\e^{tQ_p}$, where $Q_p$ is defined in \eqref{r4}. Then
the spectral radius $\mbox{Ria}( Q_{p,t})$ of $Q_{p,t}$ equals to
$\e^{-\eta_p}$. Since all coefficients of $  Q_{p,t}$ are positive,
the Perron-Frobenius theorem (see e.g. \cite[p.6]{CM}) yields that
$-\eta_p$ is a simple eigenvalue of $Q_p.$ Note that the eigenvalue
of $Q_{p,t}$ corresponding to $\e^{-\eta_p}$ is also an eigenvalue
of $Q_p$ corresponding to $-\eta_p.$ The Perron-Frobenius theorem
(see e.g. \cite[p.6]{CM}) ensures that, for $Q_p$, there exists an
eigenvector $\xi^{(p)}=(\xi_1^{(p)},\cdots,\xi_N^{(p)})\gg{\bf0}$
 corresponding to $-\eta_p.$ Now, by Lemma \ref{P-F} above,
there exists some $ p_0>0$ such that $\eta_p>0$ for any $0< p< p_0.$
Hereinafter, fix a $p$ with $0< p<1\wedge p_0$ and the corresponding
eigenvector $\xi^{(p)}\gg{\bf0}$. Then we obtain that
\begin{equation}\label{eq10}
Q_{p}\xi^{(p)}=-\eta_{ p}\xi^{(p)}\ll{\bf0}.
\end{equation}
By the It\^o formula and  $p\in(0,1\wedge p_0)$,  we
obtain from ({\bf H}) and \eqref{eq10} that
\begin{equation*}
\begin{split}
&\e^{\eta_{p} t}\E((1+|X_t^{x,i}|^2)^{
p/2}\xi_{\LL_t^i}^{(p)})\\&=(1+|x|^2)^{
p/2}\xi_{i}^{(p)}+\E\int_0^t\e^{\eta_{ p}
s}\{\eta_{p}(1+|X_s^{x,i}|^2)^{p/2}\xi_{\LL_s^i}^{(p)}+(1+|X_s^{x,i}|^2)^{p/2}(Q\xi^{(p)})(\LL_s^i)\\
&\quad+\ff{p}{2}(1+|X_s^{x,i}|^2)^{(p-2)/2}(2\<X_s^{x,i},b(X_s^{x,i},\LL_s^i)\>+\|\si(X_s^{x,i},\LL_s^i)\|^2)\xi_{\LL_s^i}^{(p)}\\
&\quad+\ff{p(p-2)}{2}(1+|X_s^{x,i}|^2)^{(p-4)/2}|\si^*(X_s^{x,i},\LL_s^i)X_s^{x,i}|^2\xi_{\LL_s^i}^{(p)}\}\d
s\\
&\le c(1+|x|^p+\e^{\eta_p t})+\E\int_0^t\e^{\rr
s}\{\eta_p\xi_{\LL_s^i}^{(p)}+(Q_p\xi^{(p)})(\LL_s^i)\}(1+|X_s^{x,i}|^2)^{p/2}\d
s\\
&= c(1+|x|^p+\e^{\eta_p t}).
\end{split}
\end{equation*}
This
implies that
\begin{equation}\label{*1}
\sup_{t\ge0}\E|X_t^{x,i}|^p<c\{1+(1+|x|^p)\e^{-\eta_pt}\}.
\end{equation}

\smallskip

Observe that $(X_t^{x,i},\LL_t^i)$ is Feller continuous (see e.g.
\cite[Theorem 2.18, p.48]{YZ}) and $B_R(0)$ is a compact subset of
$\R^n$. For arbitrary $t>0$, define a probability measure
\begin{equation*}
\mu_t(\GG):=\ff{1}{t}\int_0^tP_s(x,i;\GG)\d
s,~~~\GG\in\B(\R^n\times\mathbb{S}).
\end{equation*}
Then, for any $\vv>0$, by \eqref{*1} and Chebyshev's inequality,
there exists $R>0$ sufficiently large  such that
\begin{equation*}
\mu_t(B_R\times \mathbb{S})=\ff{1}{t}\int_0^tP(s,x,i;B_R\times
\mathbb{S})\d s\ge1-\ff{\sup_{t\ge0}\E|X_t^{x,i}|^p}{R^p}\ge1-\vv.
\end{equation*}
Hence $\{\mu_t\}_{t\ge0}$ is tight  and there exists an invariant
measure of $(X_t^{x,i},\LL_t^i)$ (see e.g. \cite[Theorem  4.14,
p.128]{Ch04}).

\smallskip
Next, we show  uniqueness of invariant measure. Again, by the It\^o
formula, it follows from \eqref{eq10} and ({\bf H}) that
\begin{equation*}
\begin{split}
\e^{\eta_{p} t}\E(|X_t^{x,i}-X_t^{y,i}|^{p}\xi_{\LL_t^i}^{(p)})
&\le|x-y|^{p}\xi_{i}^{(p)}+\E\int_0^t\e^{\eta_{p}
s}\{\eta_{p}\xi_{\LL_s^i}^{(p)}+(Q_{
p}\xi^{(p)})(\LL_s^i)\}|X_s^{x,i}-X_s^{y,i}|^{p}\d
s\\
&=|x-y|^{p}\xi_i^{(p)},
\end{split}
\end{equation*}
which gives
\begin{equation}\label{*2}
\E(|X_t^{x,i}-X_t^{y,i}|^{p})\le c|x-y|^{p}\e^{-\eta_{p}t}.
\end{equation}
Set $$\tau:=\inf\{t\ge0:\LL_t^i=\LL_t^j\}.$$ Since $\mathbb{S}$ is a
finite set, and $Q$ is irreducible, there exists $\theta>0$ such
that
\begin{equation}\label{T5}
\P(\tau>t)\le\e^{-\theta t},~~~~t>0.
\end{equation}
Observe that \eqref{*1}  holds   with different $\eta_p$ and
$\xi^{(p)}$ for any $0< p<1\wedge p_0$. For $0<p<1\wedge p_0$ above,
choose $q>1$ such that $0<pq<1\wedge p_0.$ By H\"older's inequality,
we derive from \eqref{*1}-\eqref{T5} that
\begin{equation*}
\begin{split}
\E|X_t^{x,i}-X_t^{y,j}|^{p}&=\E(|X_t^{x,i}-X_t^{y,j}|^{p}{\bf1}_{\{\tau>t\}})
+\E(|X_t^{x,i}-X_t^{y,j}|^{p}{\bf1}_{\{\tau\le
t\}})\\
&\le
(\E(|X_t^{x,i}-X_t^{y,j}|^{pq}))^{1/q}(\P(\tau>t))^{1-1/q}+\E({\bf1}_{\{\tau\le
t\}}\E(|X_t^{x,i}-X_t^{y,j}|^{p}|\F_\tau))\\
&\le
(\E(|X_t^{x,i}-X_t^{y,j}|^{pq}))^{1/q}(\P(\tau>t))^{1-1/q}+\E({\bf1}_{\{\tau\le
t\}}\E(|X_{t-\tau}^{X_{\tau}^{x, i},\LL_{\tau}^{i}}-X_{t-\tau}^{X_{\tau}^{y, j},\LL_{\tau}^{j}}|^{p})\\
&\le \e^{-\ff{q-1}{2q}\theta
t}(\E(|X_t^{x,i}-X_t^{y,j}|^{pq}))^{1/q}+c\E({\bf1}_{\{\tau\le
t\}}\e^{-\eta_p(t-\tau)}\E|X_{t-\tau}^{X_{\tau}^{x, i},\LL_{\tau}^{i}}-X_{t-\tau}^{X_{\tau}^{y, j},\LL_{\tau}^{j}}|^{p})\\
&\le c\e^{-\rr t},
\end{split}
\end{equation*}
%\begin{equation*}
%\begin{split}
%\E|X_t^{x,i}-X_t^{y,j}|^{p}&=\E(|X_t^{x,i}-X_t^{y,j}|^{p}{\bf1}_{\tau>t/2})+\E(|X_t^{x,i}-X_t^{y,j}|^{p}{\bf1}_{\tau\le
%t/2})\\
%&\le
%(\E(|X_t^{x,i}-X_t^{y,j}|^{pq}))^{1/q}(\P(\tau>t/2))^{1-1/q}+\E({\bf1}_{\tau\le
%t/2}\E(|X_t^{x,i}-X_t^{y,j}|^{p}|\F_\tau))\\
%&\le \e^{-\ff{q-1}{2q}\theta
%t}(\E(|X_t^{x,i}-X_t^{y,j}|^{pq}))^{1/q}+c\E({\bf1}_{\tau\le
%t/2}\e^{-\eta_p(t-\tau)}\E|X_\tau^{x,i}-X_\tau^{y,j}|^{p})\\
%&\le c\e^{-\rr t},
%\end{split}
%\end{equation*}
where $\rr:=\ff{(q-1)\theta}{2q}\wedge\ff{\eta_p}{2}$. Note that
\begin{equation}\label{T1}
W_p(\dd_{(x,i)}P_t,\dd_{(y,j)}P_t)\le\E|X_t^{x,i}-X_t^{y,j}|^{p}+\P(\LL_t^i\neq\LL_t^j).
\end{equation}
Assume that $\pi,\nu\in\mathcal {P}(\R^n\times\mathbb{S})$ are
invariant measures of $(X_t^{x,i},\LL_t^i)$. By
 the Kantorovich-Rubinstein duality formula (see e.g. \cite[Theorem
5.10]{V09}),  it follows from \eqref{T1} that
\begin{equation}\label{T}
\begin{split}
W_{p}(\pi,\nu)=W_{p}(\pi P_t,\nu P_t)
&=\sup_{\varphi:\mbox{Lip}(\varphi)=1}\Big\{\int_{\R^n\times\mathbb{S}}\varphi(x,i)\d(\pi
P_t)-\int_{\R^n\times\mathbb{S}}\varphi(x,i)\d(\nu P_t)\Big\}\\
&\le\int_{\R^n\times\mathbb{S}}\int_{\R^n\times\mathbb{S}}\pi(\d
x\times\{i\})\nu(\d
y\times\{j\})W_p(\dd_{(x,i)}P_t,\dd_{(y,j)}P_t)\\
&\longrightarrow0,
\end{split}
\end{equation}
where
\begin{equation*}
\mbox{Lip}(\varphi):=\sup\Big\{\ff{\varphi(x,i)-\varphi(y,j)}{d^p((x,i),(y,j))}:(x,i)\neq(y,j)\Big\}.
\end{equation*}
% Following an argument
%of  \eqref{eq10}, we deduce that there exist $p_1, \bar
%p_0,\bar\eta_{p_1}>0$ and
%$\bar\xi^{(p_1)}:=(\bar\xi_1^{(p_1)},\cdots,\bar\xi_N^{(p_1)})\gg{\bf0}$
%such that
%\begin{equation}\label{*}
%\bar Q_{ p_1}\bar\xi=-\bar\eta_{ p_1}\bar\xi\ll{\bf0},~~~0<
%p_1\le\min\{1, \bar p_0\},
%\end{equation}
%where $$\bar Q_{ p_1}:=Q+\ff{
%p_1}{2}\mbox{diag}(\bb_1,\cdots,\bb_N).$
Hence,  uniqueness of invariant measure follows.
\end{proof}

Next, we provide an example  to demonstrate that our theory is more
general than that of the existing literature.
\begin{exa}\label{Exa}
{\rm Let $\{\LL_t\}_{t\ge0}$ be a right-continuous Markov chain
taking values in $\mathbb{S}:=\{0,1\}$ with the generator
\begin{equation}\label{T2}
Q= \left(\begin{array}{ccc}
  -4 & 4\\
  \gamma & -\gamma\\
  \end{array}
  \right)
  \end{equation}
with some $\gamma>0.$  Consider a scalar Ornstein-Uhlenback (O-U)
process with regime switching
\begin{equation}\label{*5}
\d X_t=\aa_{\LL_t}X_t\d t+\si_{\LL_t}\d
W_t,~~~t>0,~~X_0=x,~\LL_0=i_0\in\mathbb{S},
\end{equation}
where $\aa_\cdot,\si_\cdot:\mathbb{S}\mapsto\R$ such that $\aa_0=1$,
and $\aa_1=-1/2$.

\smallskip

By an M-Matrix approach,  $(X_t^{x,i},\LL_t^i)$, determined by
\eqref{T2} and \eqref{*5},  has a unique invariant measure for
$\gamma\in(0,1)$ (see e.g. \cite[Example 5.1]{YM03}). It is easy to
see that the stationary distribution of $\{\LL_t\}_{t\ge0}$ is
\begin{equation*}
\mu=(\mu_0,\mu_1)=\Big(\ff{\gamma}{4+\gamma},\ff{4}{4+\gamma}\Big).
\end{equation*}
For the O-U process \eqref{*5},  $\bb_0=2$, $\bb_1=-1$ in ({\bf H})
and \eqref{eq8} holds with $\gamma\in(0,2)$. So, by Theorem
\ref{exa}, $(X_t^{x,i},\LL_t^i)$, determined by \eqref{T2} and
\eqref{*5}, admits a unique invariant measure $\pi \in\mathcal
{P}(\R\times\mathbb{S})$ for $\gamma\in(0,2)$. This means that our
result cannot be covered by the existing results. }
\end{exa}

\begin{rem}
{\rm In \eqref{eq1}, let $b,\si:\R\times\mathbb{S}\mapsto\R$,   with
$b(0,i)=\si(0,i)\equiv0$ for $i\in\mathbb{S}$, satisfy the global
Lipschitz condition, and $\{W_t\}_{t\ge0}$ be a scalar Brownian
motion. For each $i\in\mathbb{S}$, assume that there exists
$\bb_i\in\R$ such that
\begin{equation*}
2(x-y)(b(x,i)-b(y,i)) \le\bb_i|x-y|^2,~~~x,y\in\R.
\end{equation*}
By \cite[Lemma 3.2, p.120]{m08}, for $X_0^{x,i}=x\neq0$, $
\P(X_t^{x,i}\neq0 \mbox{ on } t\ge0)=1, $ i.e., almost all the
sample path of any solution starting from a non-zero state will
never reach the origin. In the sequel, without loss of generality,
we assume $X_0^{x,i}=x>0$.  Following the first part of argument for
Theorem \ref{exa}, and applying  It\^o's formula to
$\E((X_t^{x,i})^p\xi_{\LL_t^i}^{(p)})$, we can deduce that
$(X_t^{x,i},\LL_t^i)$ admits an invariant measure $\pi\in\mathcal
{P}(\R\times\mathbb{S})$ whenever $\sum_{i=1}^N\mu_i\bb_i<0$. Next,
for any $x>y$, by the comparison theorem (see e.g. \cite[Theorem
1.1, p.352]{IW89}), we have $X_t^{x,i}>X_t^{y,i}$ a.s. Then,  by
imitating the second part of argument for Theorem \ref{exa} and
utilizing It\^o's formula to
$\E((X_t^{x,i}-X_t^{y,i})^p\xi_{\LL_t^i}^{(p)}),$ uniqueness of
invariant measure for $(X_t^{x,i},\LL_t^i)$ follows provided that
$\sum_{i=1}^N\mu_i\bb_i<0$. Therefore,  for a scalar
RSDP, existence and uniqueness of invariant measure   can be
determined only by the drift coefficient in some cases. }
\end{rem}

\section{Numerical Invariant Measure}\label{sec3}
In the last section, under the ``averaging condition'' \eqref{eq8},
we discuss existence and uniqueness of invariant measure for the
semigroup generated by the pair $(X_t^{x,i},\LL_t^i)$, determined by
\eqref{eq1} and \eqref{love}.
 In this section, assuming  $N<\8$ we turn to study existence and
uniqueness of invariant measure for the
%discrete-time
semigroup generated by the EM scheme constructed as below. For a
given stepsize $\dd\in(0,1)$, we define the discrete-time EM scheme
associated with \eqref{eq1} as follows
\begin{equation}\label{eq2}
\bar Y_{(k+1)\dd}^{x,i}:=\bar Y_{k\dd}^{x,i}+b(\bar
Y_{k\dd}^{x,i},\LL_{k\dd}^i)\dd+\si(\bar
Y_{k\dd}^{x,i},\LL_{k\dd}^i)\3W_k,~k\ge0,~ \bar
Y_0^{x,i}=x,\LL_0^i=i\in\mathbb{S},
\end{equation}
where $\3W_k:=W_{(k+1)\dd}-W_{k\dd}$ stands for the  Brownian motion
increment. For convenience, we also need the following
continuous-time EM scheme
\begin{equation}\label{eq3}
Y_t^{x,i}:=x+\int_0^tb(\bar Y_{\lf s/\dd\rf\dd}^{x,i},  \LL_{\lf
s/\dd\rf\dd}^i)\d s+\int_0^t\si(\bar Y_{\lf
s/\dd\rf\dd}^{x,i},\LL_{\lf s/\dd\rf\dd}^i) \d
W_s,~t\ge0,~\LL_0^i=i\in\mathbb{S},
  \end{equation}
where, for $a\ge0,$ $\lf a\rf$ denotes the integer part of $a.$
%in which
%\begin{equation*}
%\tt Y_t:=\bar Y_{k\dd} ~~~\mbox{ and }  ~~~\tt
%\LL_t:=\LL_{k\dd}~~~\mbox{ for } t\in[k\dd,(k+1)\dd).
%\end{equation*}
Note that $ Y_{k\dd}^{x,i}=\bar Y_{k\dd}^{x,i},~k\ge0. $ That is,
the discrete-time EM scheme \eqref{eq2} coincides with the
continuous-time EM scheme \eqref{eq3} at the gridpoints whenever
they enjoy the same starting points. Hence, for some quantitative
analysis, it is sufficient to focus on $\{Y_t^{x,i}\}_{t\ge0}$
instead of $\{\bar Y_{k\dd}^{x,i}\}_{k\ge0}$.

\smallskip

Let $P_{k\dd}^\dd(x,i;\d y\times\{j\})$ be the transition probability
kernel of $(\bar Y_{k\dd}^{x,i},\LL_{k\dd}^i)$, which is a time
homogeneous Markov chain (see e.g. \cite[Theorem 6.14,
p.250]{YM06}).
%, with the starting point $(x,i).$
If
$\pi^\dd\in\mathcal {P}(\R^n\times\mathbb{S})$ satisfies
\begin{equation*}
\pi^\dd(\Gamma\times\{i\})=\sum_{j=1}^N\int_{\R^n}P_{k\dd}^\dd(x,j;\Gamma\times\{i\})\pi^\dd(\d
x\times\{j\}),~\GG\in\B(\R^n),
\end{equation*}
then we call $\pi^\dd\in\mathcal {P}(\R^n\times\mathbb{S})$ an
invariant measure of $(\bar Y_{k\dd}^{x,i},\LL_{k\dd}^i)$. Moreover,
the invariant measure $\pi^\dd\in\mathcal {P}(\R^n\times\mathbb{S})$
is also said to be a numerical invariant measure of
$(X_t^{x,i},\LL_t^i)$.
%\eqref{eq2} is said to admit an
%invariant measure if $(\bar Y_{k\dd}^{x,i},\LL_{k\dd}^i)$ does.

%\smallskip

%The lemma below also play an important role in long-time behavior of
%numerical scheme.
%\begin{lem}\label{L1.2}
%{\rm (\cite[Lemma 6.10, p.251]{YM06}) There exists a pair of
%constants $\dd_0,M>0$ such that
%\begin{equation*}
%\P(\LL_{\lfloor t/\dd\rfloor\dd}\neq\LL_t)\le M\dd ~~~\mbox{for any
%} t\ge0,~~0<\dd\le\dd_0.
%\end{equation*}
%}
%\end{lem}

\smallskip

In this section, we further assume that
$b:\R^n\times\mathbb{S}\mapsto\R^n$ and
$\si:\R^n\times\mathbb{S}\mapsto\R^n\otimes\R^m$
 are globally  Lipschitzian, i.e., for each
 $i\in\mathbb{S}$ and $x,y\in\R^n$,
 there exists an $L>0$ such that
\begin{equation}\label{eq5}
|b(x,i)-b(y,i)|+\|\si(x,i)-\si(y,i)\|\le L|x-y|.
\end{equation}
This implies the linear growth condition:
\begin{equation}\label{eq6}
|b(x,i)|+\|\si(x,i)\|\le L_0+ L|x|,~~~~x\in\R^n,
\end{equation}
where $L_0:=\max_{i\in\mathbb{S}}\{|b(0,i)|+\|\si(0,i)\|\}$.

\smallskip

In the sequel, we shall investigate existence and uniqueness of
 invariant measure for $(\bar Y_{k\dd}^{x,i},\LL_{k\dd}^i)$,
 determined by \eqref{eq2} and \eqref{love} with   additive noise
 and   multiplicative noise case respectively.

\subsection{Additive Noise Case}\label{additive}
We here consider \eqref{eq1} with additive noise in
the form
\begin{equation}\label{T3}
\d X_t=b(X_t,\LL_t)\d t+\si(\LL_t)\d
W_t,~t>0,~X_0=x\in\R^n,\LL_0=i\in\mathbb{S},
\end{equation}
where $\si:\mathbb{S}\mapsto\R^n\otimes\R^m$, and the other
quantities are defined exactly as in \eqref{eq1} and \eqref{love}.
Moreover, the EM scheme $\bar Y_{k\dd}^{x,i}$ associated with
\eqref{T3} is constructed as in \eqref{eq2} with
$\si(\cdot,\cdot)\equiv\si(\cdot)$. In what follows,
$\xi^{(p)}\gg{\bf0}$ is the eigenvector $Q_p$, defined in
\eqref{r4}, with the corresponding eigenvalue $-\eta_p<0$ for
$0<p<1\wedge p_0$, i.e., \eqref{eq10} holds.
%Hereinafter,
% $\xi^{(p)}\gg{\bf 0}$ (resp. $\bar\xi^{(p_1)}\gg{\bf 0})$ is the eigenvector of $Q_p$ (resp. $\bar Q_{p_1}$) with the
%corresponding eigenvalue $-\eta_p<0$ (resp. $-\bar\eta_{p_1}$)
 %for $0<p\le\min\{1,p_0\}$ (resp. $0<p_1\le\min\{1,\bar p_0\}$), i.e., \eqref{eq10} (resp. \eqref{*}) holds.
 Let
\begin{equation}\label{T8}
\hat{\xi}_0:=\max_{i\in\mathbb{S}}\xi_i^{(p)},~~~
\bar\xi_0:=(\min_{i\in\mathbb{S}}\xi_i^{(p)})^{-1},~~~\bb_0:=\max_{i\in\mathbb{S}}|\bb_i|,~~~q_0:=\max_{i\in\mathbb{S}}(-q_{ii}).
\end{equation}
Set
\begin{equation}\label{*11}
\aa:=
p\{\bb_0+4L^2(3+4\bb_0)+4^{(2+p)/2}\bb_0(4^{p/2}L^p+q_0\hat{\xi}_0\bar\xi_0)\}.
\end{equation}
%and
%\begin{equation}\label{d9}
%\bb:=p\{4(1+\bb_0)L^2+\bb_0\}+4^{(2+p)/2}p\bb_0(q_0\hat\xi_0\breve{\xi}_0+L^{p}).
%\end{equation}

Our main result in this subsection is as follows.

\begin{thm}\label{boun}
{\rm Let $N<\8,$ and assume further that ({\bf H}), \eqref{eq8}, and
\eqref{eq5} hold.
 Then, for
 \begin{equation}\label{*12}
 \dd<(1/(16L^2))\wedge(\eta_p/\aa)^{2/p},
 \end{equation}
 $(\bar
Y_{k\dd}^{x,i},\LL_{k\dd}^i)$ admits a unique invariant measure
$\pi^\dd\in\mathcal {P}(\R^n\times\mathbb{S}).$ }
\end{thm}

\begin{proof}
We divide the whole proof into two parts.

\smallskip

\noindent ({\bf i}) Existence of an Invariant Measure. For each
integer $q\ge1$, define the measure
\begin{equation*}
\mu_q(B_R(0)\times\mathbb{S}):=\ff{1}{q}\sum_{k=0}^q\P((\bar
Y_{k\dd}^{x,i},\LL_{k\dd}^i)\in B_R(0)\times\mathbb{S}).
\end{equation*}
%where $B_R:=\{x\in\R^n:|x|\le R\}$, a compact subset of $\R^n$, for
%some $R>0.$
To show existence of an invariant measure, it suffices to show that,
for any $(x,i)\in\R^n\times\mathbb{S},$
\begin{equation}\label{w}
\sup_{k\ge0}\E|\bar Y_{k\dd}^{x,i}|^p<\8,~~~p\in(0,1\wedge p_0).
\end{equation}
%where $0<p<1\wedge p_0$ is introduced in the argument of Theorem
%\ref{exa}.
Indeed, if so,  the Chebyshev inequality  yields
 that the measure sequence  $\{\mu_q(\cdot)\}_{q\ge1}$ is tight. Then, one can
 extract a subsequence which converges weakly to an invariant
 measure (see e.g. Meyn and Tweedie \cite{MT}).

\smallskip

In what follows, we prove that \eqref{w} holds.
%\begin{equation}\label{**}
%\sup_{k\ge0}\E|\bar Y_{k\dd}^x|^p<\8.
%\end{equation}
Let $W_{t,\dd}:=|W_t-W_{\lfloor t/\dd\rfloor\dd}|^2$. By \eqref{eq6}
and \eqref{*12}, one has
\begin{equation}\label{d5}
%\ff{2}{5}|Y_t^x|^2-c_1(\dd+W_{t,\dd})\le
|\bar Y_{\lfloor t/\dd\rfloor\dd}^{x,i}|^2\le
c(\dd+W_{t,\dd})+4|Y_t^{x,i}|^2,
\end{equation}
and
\begin{equation}\label{d6}
|Y_t^{x,i}-\bar Y_{\lfloor t/\dd\rfloor\dd}^{x,i}|^2\le
c(\dd+W_{t,\dd})+16L^2\dd|Y_t^{x,i}|^2.
\end{equation}
By applying It\^o's formula, for  any $\rr>0$ and $p\in(0,1\wedge
p_0)$, it follows from \eqref{r5}  and \eqref{eq10} that
\begin{equation*}
\begin{split}
&\e^{\rr
t}\E((1+|Y_t^{x,i}|^2)^{p/2}\xi_{\LL_t^i}^{(p)})\\
&\le(1+|x|^2)^{p/2}\xi_i^{(p)}+\E\int_0^t\e^{\rr
s}\{(\rr\xi_{\LL_s^i}^{(p)}+(Q\xi^{(p)})(\LL_s^i))(1+|Y_s^{x,i}|^2)^{p/2}\\
&\quad+\ff{p}{2}(1+|Y_s^{x,i}|^2)^{(p-2)/2}(2\<Y_s^{x,i},b(\bar
Y_{\lfloor
s/\dd\rfloor\dd}^{x,i},\LL_{\lfloor s/\dd\rfloor\dd}^i)\>+\|\si\|^2)\xi_{\LL_s^i}^{(p)}\}\d s\\
&\le(1+|x|^2)^{p/2}\xi_i^{(p)}+c\e^{\rr t}+\E\int_0^t\e^{\rr
s}\{\rr\xi_{\LL_s^i}^{(p)}+(Q\xi^{(p)})(\LL_s^i)+\ff{p}{2}\bb_{\LL_s^i}\xi_{\LL_s^i}^{(p)}\}(1+|Y_s^{x,i}|^2)^{p/2}\d s\\
&\quad+\E\int_0^t\e^{\rr s}(1+|Y_s^{x,i}|^2)^{(p-2)/2}\{\Theta_1(s)+\Theta_2(s)\}\d s\\
&\le c(1+|x|^p+\e^{\rr t})-(\eta_p-\rr)\E\int_0^t\e^{\rr
s}(1+|Y_s^{x,i}|^2)^{p/2}\xi_{\LL_s^i}^{(p)}\d
s\\
&\quad+\E\int_0^t\e^{\rr
s}(1+|Y_s^{x,i}|^2)^{(p-2)/2}\{\Theta_1(s)+\Theta_2(s)\}\d s,
\end{split}
\end{equation*}
where
\begin{equation*}
\begin{split}
\Theta_1(t)&:=p\<Y_t^{x,i}-\bar Y_{\lfloor
t/\dd\rfloor\dd}^{x,i},b(\bar Y_{\lfloor
t/\dd\rfloor\dd}^{x,i},\LL_{\lfloor
t/\dd\rfloor\dd}^i)\>\xi_{\LL_t^i}^{(p)}+\ff{p}{2}\bb_{\LL_t^i}|\bar
Y_{\lfloor
t/\dd\rfloor\dd}^{x,i}-Y_t^{x,i}|^2\xi_{\LL_t^i}^{(p)}\\
&\quad+p\bb_{\LL_t^i}\<Y_t^{x,i},\bar Y_{\lfloor
t/\dd\rfloor\dd}^{x,i}-Y_t^{x,i}\>\xi_{\LL_t^i}^{(p)},\\
\end{split}
\end{equation*}
and
\begin{equation*}
\begin{split}
 \Theta_2(t)&:=\ff{p}{2}(\bb_{\LL_{{\lfloor
t/\dd\rfloor\dd}}^i}-\bb_{\LL_t^i})|\bar Y_{\lfloor
t/\dd\rfloor\dd}^{x,i}|^2\xi_{\LL_t^i}^{(p)}.
\end{split}
\end{equation*}
By the fundamental inequality: $a^\nu b^{1-\nu}\le\nu a+(1-\nu)b$
with $a,b>0$ and $\nu\in(0,1),  $  \eqref{eq6}, \eqref{d5} and
\eqref{d6} yield that
\begin{equation}\label{*13}
\begin{split}
\Theta_1(t) &\le \ff{p}{2}\{
(\bb_0+(1+\bb_0)(\ss\dd)^{-1}\}|Y_t^{x,i}-\bar
Y_{\lfloor t/\dd\rfloor\dd}^{x,i}|^2\xi_{\LL_t^i}^{(p)}\\
&\quad+\ff{p\bb_0}{2}\ss\dd|Y_t^{x,i}|^2\xi_{\LL_t^i}^{(p)}+\ff{p}{2}\ss\dd|b(\bar
Y_{\lfloor t/\dd\rfloor\dd}^{x,i},\LL_{\lfloor
t/\dd\rfloor\dd})|^2\xi_{\LL_t^i}^{(p)}\\
&\le c(1+W_{t,\dd})+c\{
(\bb_0+(1+\bb_0)(\ss\dd)^{-1}\}(\dd+W_{t,\dd})\xi_{\LL_t^i}^{(p)}\\
&\quad+p\{\bb_0+4L^2(3+4\bb_0)\}\ss\dd|Y_t^{x,i}|^2\xi_{\LL_t^i}^{(p)}.
\end{split}
\end{equation}
For any $t\le\dd$, due to $q_{ii}<0$ one has
\begin{equation*}
\P(\LL_t^i\neq\LL_0^i=i)=1-\P(\LL_t^i=\LL_0^i)\le1-\e^{q_{ii}t}\le1-\e^{q_{ii}\dd}\le-q_{ii}\dd\le
q_0\dd.
\end{equation*}
This further gives that
\begin{equation}\label{T4}
\begin{split}
\E(|Y_{\lfloor t/\dd\rfloor\dd}^{x,i}|^p{\bf1}_{\LL_{\lfloor
t/\dd\rfloor\dd}^i\neq\LL_t^i})&=\E(\E(|Y_{\lfloor
t/\dd\rfloor\dd}^{x,i}|^p{\bf1}_{\LL_t^i\neq\LL_{\lfloor
t/\dd\rfloor\dd}^i}|\F_{{\lfloor t/\dd\rfloor\dd}}))\\
&=\E(|Y_{\lfloor
t/\dd\rfloor\dd}^{x,i}|^p\E({\bf1}_{\LL_t^i\neq\LL_{\lfloor
t/\dd\rfloor\dd}^i}|\F_{{\lfloor t/\dd\rfloor\dd}}))\\
&=\E(|Y_{\lfloor
t/\dd\rfloor\dd}^{x,i}|^p\E({\bf1}_{\LL_t^i\neq\LL_{\lfloor
t/\dd\rfloor\dd}^i}|\LL_{{\lfloor t/\dd\rfloor\dd}}^i))\\
&\le q_0\dd\E|Y_{\lfloor t/\dd\rfloor\dd}^{x,i}|^p,
\end{split}
\end{equation}
where we have used that $\{W_t\}_{t\ge0}$ is independent of
$\{\LL_t\}_{t\ge0}$. Furthermore, in light of \eqref{d5}, \eqref{d6}
and \eqref{T4}, it follows that
\begin{equation}\label{*14}
\begin{split}
&\E\int_0^t\e^{\rr s}(1+|Y_s^{x,i}|^2)^{(p-2)/2}\Theta_2(s)\d s\\
&\le p\bb_0\E\int_0^t\e^{\rr s}(1+|Y_s^{x,i}|^2)^{(p-2)/2}|\bar
Y_{\lfloor s/\dd\rfloor\dd}^{x,i}|^2{\bf1}_{\LL_s^i\neq\LL_{\lfloor
s/\dd\rfloor\dd}^i}\xi_{\LL_s^i}^{(p)}\d s\\
&\le c\e^{\rr t}+4p\bb_0\hat{\xi}_0\E\int_0^t\e^{\rr s}(1+|\bar
Y_{\lfloor
s/\dd\rfloor\dd}^{x,i}|^2)^{p/2}{\bf1}_{\LL_s^i\neq\LL_{\lfloor
s/\dd\rfloor\dd}^i}\d s\\
&\quad+4p\bb_0\E\int_0^t\e^{\rr s}(1+|Y_s^{x,i}-\bar
Y^{x,i}_{\lfloor
s/\dd\rfloor\dd}|^2)^{p/2}{\bf1}_{\LL_s^i\neq\LL_{\lfloor
s/\dd\rfloor\dd}^i}\xi_{\LL_s^i}^{(p)}\d s\\
&\le c\e^{\rr
t}+4^{(2+p)/2}p\bb_0(4^{p/2}L^p+q_0\hat{\xi}_0\bar\xi_0)\dd^{p/2}\int_0^t\e^{\rr
s}\E((1+|Y_s^{x,i}|^2)^{p/2}\xi_{\LL_s^i}^{(p)})\d s.
\end{split}
\end{equation}
Consequently, according to \eqref{*13} and \eqref{*14}, we arrive at
\begin{equation*}
\begin{split}
\e^{\rr t}\E((1+|Y_t^x|^2)^{p/2}\xi_{\LL_t^i}^{(p)}) &\le
c(1+|x|^p+\e^{\rr t})-(\eta_p-\rr-\aa\dd^{p/2})\int_0^t\e^{\rr
s}\E((1+|Y_s|^2)^{p/2}\xi_{\LL_s^i}^{(p)})\d s,
\end{split}
\end{equation*}
where $\aa>0$ is defined in \eqref{*11}. Taking
$\rr=\eta_p-\aa\dd^{p/2}>0$ due to \eqref{*12} leads to
 \eqref{w}.

\smallskip

\noindent ({\bf ii})Uniqueness of Invariant Measure. By checking the
second part of argument for Theorem \ref{existence}, we need only to
show that
\begin{equation}\label{T6}
\begin{split}
\E|\bar Y_{k\dd}^{x,i}-\bar Y_{k\dd}^{y,i}|^p\le c|x-y|^p\e^{-\eta_p
t},~~~x,y\in\R^n.
\end{split}
\end{equation}
For $\dd\in(0,1)$ such that \eqref{*12} holds, note that
\begin{equation}\label{d8}
\begin{split}
|\bar Y_{\lfloor t/\dd\rfloor\dd}^{x,i}-\bar Y_{\lfloor
t/\dd\rfloor\dd}^{y,i}|^2\le4|Y_t^{x,i}-Y_t^{y,i}|^2,
\end{split}
\end{equation}
and
\begin{equation}\label{d2}
|Y_t^{x,i}-Y_t^{y,i}-(\bar Y_{\lfloor t/\dd\rfloor\dd}^{x,i}-\bar
Y_{\lfloor t/\dd\rfloor\dd}^{y,i})|^2
\le4L^2\dd|Y_t^{x,i}-Y_t^{y,i}|^2.
\end{equation}
For arbitrary $\vv>0$, $\rr>0$, and $p\in(0,1\wedge p_0)$, by the
It\^o formula and ({\bf H}), it follows from \eqref{eq10}  that
\begin{equation*}
\begin{split}
\E&(\e^{\rr t}(\vv+|Y_t^{x,i}-Y_t^{y,i}|^2)^{p/2}\xi_{\LL_t^i}^{(p)})\\
&\le c(\vv^{p/2}+|x-y|^p)+\E\int_0^t\e^{\rr
s}(\vv+|Y_s^{x,i}-Y_s^{y,i}|^2)^{(p-2)/2}\{(\vv+|Y_s^{x,i}-Y_s^{y,i}|^2)(\rr\xi_{\LL_s^i}^{(p)}+(Q\xi^{(p)})(\LL_s^i))\\
&\quad+p\<Y_s^{x,i}-Y_s^{y,i},b(\bar Y_{\lfloor
s/\dd\rfloor\dd}^{x,i}, \LL_{\lfloor s/\dd\rfloor\dd}^i)-b(\
Y_{\lfloor s/\dd\rfloor\dd}^{y,i},
\LL_{\lfloor s/\dd\rfloor\dd}^i)\>\xi_{\LL_s^i}^{(p)}\d s\\
&\le c(\vv^{p/2}+|x-y|^p)+\E\int_0^t\e^{\rr
s}(\vv+|Y_s^{x,i}-Y_s^{y,i}|^2)^{p/2}\{\rr\xi_{\LL_s^i}^{(p)}+(Q\xi^{(p)})(\LL_s^i)+\ff{p}{2}\bb_{\LL_s^i}\xi_{\LL_s^i}^{(p)}\}\d s\\
&\quad+\E\int_0^t\e^{\rr
s}(\vv+|Y_s^{x,i}-Y_s^{y,i}|^2)^{(p-2)/2}\{\Upsilon_1(s)+\Upsilon_2(s)\}\d
s\\
&\le c(\vv^{p/2}+|x-y|^p)-(\eta_{p}-\rr)\E\int_0^t\e^{\rr
s}(\vv+|Y_s^{x,i}-Y_s^{y,i}|^2)^{p/2}\xi_{\LL_s^i}^{(p)}\d s\\
\end{split}
\end{equation*}
\begin{equation*}
+\E\int_0^t\e^{\rr
s}(\vv+|Y_s^{x,i}-Y_s^{y,i}|^2)^{(p-2)/2}\{\Upsilon_1(s)+\Upsilon_2(s)\}\d
s,~~~~~~~~~~~~~~~~~~~~~~~~~~~~~~~~~~~~~~~~~~~~
\end{equation*}
where
\begin{equation*}
\begin{split}
\Upsilon_1(t)&:=p\<Y_t^{x,i}-Y_t^{y,i}-(\bar Y_{\lfloor
t/\dd\rfloor\dd}^{x,i}-\bar Y_{\lfloor
t/\dd\rfloor\dd}^{y,i}),b(\bar Y_{\lfloor t/\dd\rfloor\dd}^{x,i},
\LL_{\lfloor t/\dd\rfloor\dd}^i)-b(\bar Y_{\lfloor
t/\dd\rfloor\dd}^{y,i}, \LL_{\lfloor
t/\dd\rfloor\dd}^i)\>\xi_{\LL_t^i}^{(p)},\\
&\quad+\ff{p}{2}\bb_{\LL_t^i}|\bar Y_{\lfloor
t/\dd\rfloor\dd}^{x,i}-\bar Y_{\lfloor
t/\dd\rfloor\dd}^{y,i}-(Y_t^{x,i}-
Y_t^{y,i})|^2\xi_{\LL_t^i}^{(p)},\\
&\quad+p\bb_{\LL_t^i}\<Y_t^{x,i}- Y_t^{y,i},\bar Y_{\lfloor
t/\dd\rfloor\dd}^{x,i}-\bar Y_{\lfloor t/\dd\rfloor\dd}^{y,i}-(Y_t^{x,i}- Y_t^{y,i})\>\xi_{\LL_t^i}^{(p)}\\
\Upsilon_2(t)&:=\ff{p}{2}(\bb_{\LL_{\lfloor
t/\dd\rfloor\dd}^i}-\bb_{\LL_t}^i)|\bar Y_{\lfloor
t/\dd\rfloor\dd}^{x,i}-\bar Y_{\lfloor
t/\dd\rfloor\dd}^{y,i}|^2\xi_{\LL_t^i}^{(p)}.
\end{split}
\end{equation*}
Observe from \eqref{eq5}, \eqref{d8} and \eqref{d2} that
\begin{equation*}
\begin{split}
\Upsilon_1(t) &\le
\ff{p}{2}\{\bb_0+(1+\bb_0)(\ss\dd)^{-1}\}|Y_t^{x,i}-Y_t^{y,i}-(\bar
Y_{\lfloor t/\dd\rfloor\dd}^{x,i}-\bar Y_{\lfloor
t/\dd\rfloor\dd}^{y,i})|^2\xi_{\LL_t^i}^{(p)}\\
&\quad+\ff{pL^2}{2}\ss\dd|\bar Y_{\lfloor
t/\dd\rfloor\dd}^{x,i}-\bar Y_{\lfloor
t/\dd\rfloor\dd}^{y,i}|^2\xi_{\LL_t^i}^{(p)}+\ff{p\bb_0}{2}\ss\dd|Y_t^{x,i}-Y_t^{y,i}|^2\xi_{\LL_t^i}^{(p)}\\
&\le
p\{4(1+\bb_0)L^2+\bb_0\}\ss\dd|Y_t^{x,i}-Y_t^{y,i}|^2\xi_{\LL_t^i}^{(p)}.
\end{split}
\end{equation*}
As \eqref{*14} was done, by virtue of \eqref{eq5}, \eqref{d8}, and
\eqref{d2}, we deduce that
\begin{equation*}
\begin{split}
&\E\int_0^t\e^{\rr
s}(\vv+|Y_s^{x,i}-Y_s^{y,i}|^2)^{(p-2)/2}\Upsilon_2(s)\d
s\\
&\le
c\vv^{p/2}+4^{(2+p)/2}p\bb_0(q_0\hat\xi_0\bar\xi_0+L^{p})\dd^{p/2}\E\int_0^t\e^{\rr
s}(\vv+|Y_s^{x,i}-Y_s^{y,i}|^2)^{p/2}\xi_{\LL_s^i}^{(p)}\d s.
\end{split}
\end{equation*}
As a consequence, we arrive at
\begin{equation*}
\begin{split}
\E&(\e^{\rr t}(\vv+|Y_t^{x,i}-Y_t^{y,i}|^2)^{p/2}\xi_{\LL_t^i}^{(p)})\\
&\le c\vv^{p/2}-(\eta_{p}-\rr-\aa\dd^{p/2})\E\int_0^t\e^{\rr
s}(\vv+|Y_s^{x,i}-Y_s^{y,i}|^2)^{p/2}\xi_{\LL_s^i}^{(p)}\d s,
\end{split}
\end{equation*}
where $\aa>0$ is defined as in \eqref{*11}. Then \eqref{T6} follows
by choosing $\rr=\eta_{p}-\aa\dd^{p/2}>0$ own to \eqref{*12} and
taking $\vv\downarrow0$.
\end{proof}

\begin{rem}\label{re3.1}
{\rm Let us reexamine the Example \ref{Exa}. Note that   ({\bf H}),
 \eqref{eq8}, and \eqref{eq5} hold  with     $\bb_1=1$,
$\bb_2=-\ff{1}{2}$,  $\gamma\in(0,2)$, and  $L=1$  respectively. So,
 $(\bar Y_{k\dd}^{x,i},\LL_{k\dd}^i)$, associated with \eqref{*5} and
\eqref{T2}, admits a unique invariant measure $\pi^\dd\in \mathcal
{P}(\R\times\mathbb{S})$
   whenever the stepsize
$
 \dd<(1/16)\wedge(\eta_p/\aa)^{2/p}
$
for $\aa= p\{29+4^{(2+p)/2}(4^{p/2}+4\hat{\xi}_0\bar\xi_0)\}$. }
\end{rem}

The following theorem reveals that   numerical invariant measure
$\pi^\dd$ converges in the  Wasserstein distance to the underlying
one.

\begin{thm}\label{W_p}
{\rm Under the assumptions of Theorem \ref{boun},  there exists $c>0$ such that
\begin{equation*}
W_p(\pi,\pi^\dd)\le c\dd^{p/2},~~~~p<1\wedge p_0,
\end{equation*}
where $p_0>0$ is introduced in the argument of Theorem
\ref{existence}. }
\end{thm}

\begin{proof}
For any $p<1\wedge p_0$, note that
\begin{equation*}
W_p(\dd_{(x,i)}P_{k\dd},\pi)\le\int_{\R^n\times\mathbb{S}}\pi(\d
y\times\{j\})W_p(\dd_{(x,i)}P_{k\dd},\dd_{(y,j)}P_{k\dd}),
\end{equation*}
and
\begin{equation*}
W_p(\dd_{(x,i)}P_{k\dd}^\dd,\pi^\dd)\le\int_{\R^n\times\mathbb{S}}\pi^\dd(\d
y\times\{j\})W_p(\dd_{(x,i)}P_{k\dd}^\dd,\dd_{(y,j)}P_{k\dd}^\dd).
\end{equation*}
Then, by a close inspection of arguments for Theorem \ref{exa} and
\ref{existence}, for $\dd\in(0,1)$such that \eqref{*12}, there exist
$k>0$ sufficiently large and $c_1>0$ such that
\begin{equation*}
W_p(\dd_{(x,i)}P_{k\dd},\pi)+W_p(\dd_{(x,i)}P_{k\dd}^\dd,\pi^\dd)\le
c_1\dd^{p/2}.
\end{equation*}
Moreover, for fixed $k>0$ above, it follows from  \cite[Theorem
3.1]{YM04} that
\begin{equation*}
W_p(\dd_{(x,i)}P_{k\dd},\dd_{(x,i)}P_{k\dd}^\dd)\le c_2\dd^{p/2}
\end{equation*}
for some $c_2>0.$ Then the desired assertion follows from the
triangle inequality.
\end{proof}

\subsection{Multiplicative Noise Case}\label{sec3.1}
In the previous subsection, we discuss existence and uniqueness of
numerical invariant measures for the RSDP  \eqref{eq1} and
\eqref{love} with additive noise. While, in this subsection, we turn
to study the case of multiplicative noise. We further assume that
\begin{equation}\label{*6}
\min_{i\in\mathbb{S},\bb_i>0}\{-q_{ii}/\bb_i,\}>1.
\end{equation}
Under this condition, by Lemma \ref{P-F} (ii) we can take $p=2$ in \eqref{r4}.
Set
\begin{equation*}
\mbox{diag}(\bb):=\mbox{diag}(\bb_1,\cdots,\bb_N),~~~~
Q_2:=Q+\mbox{diag}(\bb),~~\eta_2:=-\max_{\gamma\in\mbox{spec}(
Q_2)}\mbox{Re}\gamma,
\end{equation*}
where $Q$ is the $Q$-matrix of $\{\Lambda_t\}_{t\ge0},$ and
$\mbox{spec}(Q_2)$ denotes the spectrum of $Q_2.$ Following an
argument of \eqref{eq10}, we can deduce from \eqref{eq8} and
\eqref{*6}
 that
 there exists an
eigenvector
 $\xi^{(2)}\gg{\bf0}$  of $Q_2$  with
eigenvalue $-\eta_2<0$  such that
\begin{equation}\label{T7}
Q_{2}\xi^{(2)}=-\eta_{ 2}\xi^{(2)}\ll{\bf0}.
\end{equation}
Set
\begin{equation}\label{*8}
\hat\xi_2:=\max_{i\in\mathbb{S}}\xi_i^{(2)},
~~\bar\xi_2:=\min_{i\in\mathbb{S}}(\xi_i^{(2)})^{-1},~~
 \bb:=\{(1+12q_0)\bb_0+8L^2(5+6\bb_0)\}\hat\xi_2\bar\xi_2,
\end{equation}
where $q_0,\bb_0>0$ are defined as in \eqref{T8}.

\smallskip

Our main result in this subsection is presented as follows.

\begin{thm}\label{existence}
{\rm  Let $N<\8$, and assume further that ({\bf H}), \eqref{eq8},
\eqref{eq5}, and \eqref{*6} hold. For
\begin{equation}\label{c2}
 \dd<(1/(32L^2))\wedge(\eta_2/\bb)^2,
 \end{equation}
where $\bb>0$ is given in \eqref{*8},   $(\bar
Y_{k\dd}^{x,i},\LL_{k\dd}^i)$  admits a unique measure
$\pi^\dd\in\mathcal {P}(\R^n\times\mathbb{S}).$ }
\end{thm}

\begin{proof}
The ideas of argument for Theorem \ref{existence} is analogous to
that of Theorem \ref{boun}. However, we herein give an outline of
the argument to point out some corresponding differences.

\smallskip
\noindent ({\bf i}) Existence of an Invariant Measure. To end this,
it is sufficient to show that
\begin{equation}\label{ww}
\sup_{k\ge0}\E|\bar
Y_{k\dd}^{x,i}|^2<\8,~~~(x,i)\in\R^n\times\mathbb{S}.
\end{equation}
From \eqref{eq3} and \eqref{eq6}, one has
\begin{equation*}
\begin{split}
\E |Y_t^{x,i}-\bar Y_{\lfloor t/\dd\rfloor\dd}^{x,i}|^2\le
c\dd+8L^2\dd\E|Y_{\lfloor t/\dd\rfloor\dd}^{x,i}|^2.
\end{split}
\end{equation*}
This further leads to
  \begin{equation*}
\begin{split}
\E |\bar Y_{\lfloor t/\dd\rfloor\dd}^{x,i}|^2
\le 2 \E|Y_t^{x,i}|^2+2 \E|Y_t^{x,i}-\bar Y_{\lfloor
t/\dd\rfloor\dd}^{x,i}|^2\le 2 \E|Y_t^{x,i}|^2+2(c+8L^2\E
|Y_{\lfloor t/\dd\rfloor\dd}^{x,i}|^2)\dd.
\end{split}
\end{equation*}
Hence, due to \eqref{c2},
\begin{equation}\label{d4}
\E |\bar Y_{\lfloor t/\dd\rfloor\dd}^{x,i}|^2\le
c+4\E|Y_t^{x,i}|^2~~\mbox{ and }~~\E |Y_t^{x,i}-\bar Y_{\lfloor
t/\dd\rfloor\dd}^{x,i}|^2\le c\dd+32L^2\dd\E|Y_t^{x,i}|^2.
\end{equation}
By It\^o's formula,
for any $\rr>0$, it follows from  \eqref{r5} and \eqref{T7} that
\begin{equation}\label{w2}
\begin{split}
&\e^{\rr t}\E(|Y_t^{x,i}|^2\xi_{\LL_t^i}^{(2)}) \\&=
|x|^2\xi_i^{(2)}+\E\int_0^t\e^{\rr
s}\{\rr| Y_s^{x,i}|^2\xi_{\LL_s^i}^{(2)}+|Y_s^{x,i}|^2(Q\xi^{(2)})(\LL_s^i)\\
&\quad+(2\<Y_s^{x,i},b(\bar Y_{\lfloor
s/\dd\rfloor\dd}^{x,i},\LL_{\lfloor s/\dd\rfloor\dd}^i)\>+\|\si(\bar
Y_{\lfloor s/\dd\rfloor\dd}^{x,i},\LL_{\lfloor
s/\dd\rfloor\dd}^i)\|^2)\xi_{\LL_s^i}^{(2)}\}\d s\\
&\le c(|x|^2 + \e^{\rr t})+\E\int_0^t\e^{\rr
s}\{\rr\xi_{\LL_s^i}^{(2)}+(Q\xi^{(2)})(\LL_s^i)+\bb_{\LL_s^i}\xi_{\LL_s^i}^{(2)}\}|Y_s^{x,i}|^2\d s\\
&\quad+\int_0^t\e^{\rr
s}\{\Lambda_1(s)+\Lambda_2(s)\}\d s\\
&\le c(|x|^2 + \e^{\rr t})-\int_0^t\e^{\rr s}(\eta_2-\rr)
\E(|Y_s^{x,i}|^2\xi_{\LL_s^i}^{(2)})\d s+\int_0^t\e^{\rr
s}\{\Lambda_1(s)+\Lambda_2(s)\}\d s,
\end{split}
\end{equation}
where
\begin{equation}\label{T9}
\begin{split}
\Lambda_1(t)&:=2\E(\<Y_t^{x,i}-\bar Y_{\lfloor
t/\dd\rfloor\dd}^{x,i},b(\bar Y_{\lfloor
t/\dd\rfloor\dd}^{x,i},\LL_{\lfloor
t/\dd\rfloor\dd}^i)\>\xi_{\LL_t^i}^{(2)})+\E(\bb_{\LL_t^i}|Y_{\lfloor
t/\dd\rfloor\dd}^{x,i}-Y_t^{x,i}|^2\xi_{\LL_t^i}^{(2)})\\
&\quad+2\E(\bb_{\LL_t^i}\<Y_t^{x,i},Y_{\lfloor
t/\dd\rfloor\dd}^{x,i}-Y_t^{x,i}\>\xi_{\LL_t^i}^{(2)}),\\
\Lambda_2(t)&:=\E((\bb_{\LL_{\lfloor
t/\dd\rfloor\dd}^i}-\bb_{\LL_t^i})|Y_{\lfloor
t/\dd\rfloor\dd}^{x,i}|^2\xi_{\LL_t^i}^{(2)}).
\end{split}
\end{equation}
From \eqref{eq6}  and \eqref{d4},  we derive  that
\begin{equation*}
\begin{split}
\Lambda_1(t)
&\le\{\bb_0+(1+\bb_0)(\ss\dd)^{-1})\}\hat\xi_2\E|Y_t^x-Y_{\lfloor
t/\dd\rfloor\dd}^x|^2\\
&\quad+\hat\xi_2\ss\dd\E|b(\bar Y_{\lfloor
t/\dd\rfloor\dd}^x,\LL_{\lfloor
t/\dd\rfloor\dd}^i)|^2+\bb_0\hat\xi_2\ss\dd\E|Y_t^x|^2\\
&\le
c+\{\bb_0+8L^2(5+4\bb_0)\}\hat\xi_2\bar\xi_2\ss\dd\E(|Y_t^x|^2\xi_{\LL_t^i}^{(2)}).
\end{split}
\end{equation*}
Next, by \eqref{T4} with $p=2$ and \eqref{d4}, we have
\begin{equation*}
\begin{split}
\Lambda_2(t)&\le 2\bb_0\hat\xi_2\E(|Y_{\lfloor
t/\dd\rfloor\dd}^{x,i}|^2{\bf1}_{\LL_t^i\neq\LL_{\lfloor
t/\dd\rfloor\dd}^i})\le
c+8q_0\bb_0\hat\xi_2\bar\xi_2\dd\E(|Y_t^{x,i}|^2\xi_{\LL_t^i}^{(2)}).
\end{split}
\end{equation*}
Thus, we arrive at
\begin{equation*}
\begin{split}
&\e^{\rr t}\E((1+|Y_t^{x,i}|^2)^{p/2}\xi_{\LL_t^i}^{(2)})\le
c(|x|^2+\e^{\rr t})-\E\int_0^t\e^{\rr
s}(\eta_2-\rr-\bb\sqrt{\dd})(1+|Y_s^{x,i}|^2)^{p/2}\xi_{\LL_s^i}^{(2)}\d
s.
\end{split}
\end{equation*}
Taking $\rr=\eta_2-\aa\sqrt{\dd}>0$ thanks to \eqref{c2} yields the
desired assertion \eqref{ww}.

\smallskip

\noindent({\bf ii}) Uniqueness of Invariant Measure.  We need only
to show that
\begin{align}\label{w1}
\E(Y_t^{x,i}-Y_t^{y,i}|^2)\le c\e^{-\rr t}  |x-y|^2.
\end{align}
For  $\dd\in(0,1)$ such that \eqref{c2}, we deduce from \eqref{eq3}
and \eqref{eq5} that
\begin{equation}\label{d7}
\begin{split}
\E |\bar Y_{\lfloor t/\dd\rfloor\dd}^{x,i}-\bar Y_{\lfloor
t/\dd\rfloor\dd}^{y,i}|^2\le6\E|Y_t^{x,i}-Y_t^{y,i}|^2,
\end{split}
\end{equation}
and
\begin{equation}\label{d1}
\begin{split}
\E|Y_t^{x,i}-Y_t^{y,i}-(\bar Y_{\lfloor t/\dd\rfloor\dd}^{x,i}-\bar
Y_{\lfloor t/\dd\rfloor\dd}^{y,i})|^2&\le 2
4L^2\dd\E|Y_t^{x,i}-Y_t^{y,i}|^2.
\end{split}
\end{equation}
For any $\rr>0$, by It\^o's formula  and \eqref{r6}, it follows from
\eqref{T7} that
\begin{equation*}
\begin{split}
&\E(\e^{\rr
t}|Y_t^{x,i}-Y_t^{y,i}|^2\xi_{\LL_t^i}^{(2)})\\
&\le|x-y|^2\xi_i^{(2)}-(\eta_2-\rr)\int_0^t\e^{\rr s}
\E(|Y_s^{x,i}-Y_s^{y,i}|^2\xi_{\LL_s^i}^{(2)})\d s+\int_0^t\e^{\rr
s}\Gamma(s)\d s,
\end{split}
\end{equation*}
where
\begin{equation*}
\begin{split}
\Gamma(t) &=2\E(\<Y_t^{x,i}-Y_t^{y,i}-( \bar Y_{\lfloor
t/\dd\rfloor\dd}^{x,i}-\bar Y_{\lfloor t/\dd\rfloor\dd}^{y,i}),
b(\bar Y_{\lfloor t/\dd\rfloor\dd}^{x,i}, \LL_{\lfloor
t/\dd\rfloor\dd}^i)-b(\bar Y_{\lfloor t/\dd\rfloor\dd}^{y,i},
\LL_{\lfloor
t/\dd\rfloor\dd}^i)\>\xi_{\LL_t^i}^{(2)})\\
 &\quad+2\E(\<Y_t^{x,i}-Y_t^{y,i},\bar Y_{\lfloor t/\dd\rfloor\dd}^{x,i}-\bar
Y_{\lfloor
t/\dd\rfloor\dd}^{y,i}-(Y_t^x-Y_t^y)\>\bb_{\LL_t^i}\xi_{\LL_t^i}^{(2)})\\
 &\quad+\E(|\bar Y_{\lfloor t/\dd\rfloor\dd}^{x,i}-\bar Y_{\lfloor
t/\dd\rfloor\dd}^{y,i}-(Y_t^{x,i}-Y_t^{y,i})|^2\bb_{\LL_t^i}\xi_{\LL_t^i}^{(2)})\\
&\quad+\E((\bb_{\LL_{{\lfloor
t/\dd\rfloor\dd}}^i}-\bb_{\LL_t^i})|\bar Y_{\lfloor
t/\dd\rfloor\dd}^{x,i}-\bar Y_{\lfloor
t/\dd\rfloor\dd}^{y,i}|^2\xi_{\LL_t^i}^{(2)}).
\end{split}
\end{equation*}
Notice from \eqref{eq5},  \eqref{d7} and \eqref{d1} that
\begin{equation*}
\begin{split}
\Gamma(t)
&\le
\bb\ss\dd\E(| Y_t^{x,i}-Y_t^{y,i}|^2\xi_{\LL_t^i}^{(2)}),
\end{split}
\end{equation*}
in which $\bb>0$ is defined  in \eqref{*8}. Consequently, we have
\begin{equation*}
\begin{split}
\E(\e^{\rr t}|Y_t^{x,i}-Y_t^{y,i}|^2\xi_{\LL_t^i}^{(2)})&\le
|x-y|^2\xi_i^{(2)}-(\eta_2-\rr-\bb\ss\dd)\int_0^t\e^{\rr s}
\E(|Y_s^{x,i}-Y_s^{y,i}|^2\bar\xi_{\LL_s^i}^{(2)})\d s,
\end{split}
\end{equation*}
 Choosing  $\rr=(\eta_2-\bb\sqrt{\dd})>0$ due to \eqref{c2}
 leads to \eqref{w1}.
\end{proof}

Now we construct an example to show an application of Theorem
\ref{existence}.

\begin{exa}\label{Exa1}
{\rm Let $\{\LL_t\}_{t\ge0}$ be a right-continuous Markov chain
taking values in $\mathbb{S}:=\{1,2,3\}$ with the generator
\begin{equation*}
Q= \left(\begin{array}{ccc}
  -(3+\nu) & \nu & 3\\
  1 & -3 & 2\\
  1 &2 &-3
  \end{array}
  \right)
  \end{equation*}
for some $\nu\ge0.$  Consider a scalar linear SDE with regime
switching
\begin{equation}\label{**5}
\d X_t=\aa_{\LL_t}X_t\d t+\si_{\LL_t}X_t\d
W_t,~~~t\ge0,~~X_0=x,~\LL_0=i_0,
\end{equation}
where $\aa_\cdot,\si_\cdot:\mathbb{S}\mapsto\R$ such that
$$\aa_1=\ff{1}{2},\aa_2=-2,\aa_3=-3,~~~~\si_1=\ff{1}{3},\si_2=2,\si_3=1.$$

\smallskip
Observe that \eqref{eq5}   holds  with $L=4$, and ({\bf H}) holds
for   $\bb_1=\ff{10}{9}, \bb_2=0$, and $\bb_3=-5$. Since the Markov
chain possesses the stationary distribution
\begin{equation*}
\mu=(\mu_1,\mu_2,\mu_3)=\Big(\ff{5}{20+5\nu},\ff{6+3\nu}{20+5\nu},\ff{9+2\nu}{20+5\nu}\Big).
\end{equation*}
Note that the solution of the equation
$$
\d X_t^{(1)}=\aa_{\LL_t}X_t^{(1)}\d t+\si_{\LL_t}X_t^{(1)}\d
W_t
$$
with $X_0^{(1)}\neq 0$ will explode to $\infty$ with probability
one. However it is easy to see that \eqref{eq8} and \eqref{*6} are
satisfied respectively for any $\nu\ge0$. Then,  $(\bar
Y_{k\dd}^{x,i},\LL_{k\dd}^i)$  has a unique invariant measure for
sufficiently small $\dd\in(0,1).$
 }
\end{exa}

We can also obtain the following convergence  rate of numerical
invariant measure.

\begin{thm}\label{W1}
{\rm Under   the assumptions of Theorem \ref{existence},  for
$\dd\in(0,1)$ such that \eqref{c2}, there exists $c>0$ such that
\begin{equation*}
W_1(\pi,\pi^\dd)\le c\dd^{1/2}.
\end{equation*}}
\end{thm}

\begin{proof}
We omit the proof of Theorem \ref{W1} since it is  similar to
that of Theorem \ref{W_p}.

\end{proof}

\section{Numerical Invariant Measure:  Reversible Case}\label{sec4}
In the last section, we investigate existence and uniqueness of
numerical invariant measures for RSDPs with additive noises and
multiplicative noises respectively, where the Markov chain
$\{\LL_t\}_{t\ge0}$ need not to be reversible, i.e.,
$\pi_iq_{ij}=\pi_jq_{ji},i,j\in\mathbb{S}$, for some probability
measure $\pi:=(\pi_1,\cdots,\pi_N)$. While, throughout this section,
we shall always assume that the Markov chain $\{\LL_t\}_{t\ge0}$,
with the state space $\mathbb{S}:=\{1,\cdots,N\}$, $N<\8,$ is
reversible with the probability measure $\pi$ above. For such case,
under a new condition  we  study existence and
uniqueness of numerical invariant measure   for  multiplicative
noise case.

\smallskip

To begin with, we need to introduce some notation.
 Let
\begin{equation*}
L^2(\pi):=\Big\{f\in\B(\mathbb{S}):\sum_{i=1}^N\pi_if_i^2<\8\Big\}.
\end{equation*}
Then $(L^2(\pi),\<\cdot,\cdot\>_0,\|\cdot\|_0)$ is a Hilbert space
with the inner product $\<f,g\>_0:=\sum_{i=1}^N\pi_if_ig_i,f,g\in
L^2(\pi)$. Define the bilinear form $(D(f),\mathscr{D}(D))$ as
\begin{equation*}
D(f):=\ff{1}{2}\sum_{i,j=1}^N\pi_iq_{ij}(f_j-f_i)^2-\sum_{i=1}^N\pi_i\bb_if_i^2,~~~f\in
L^2(\pi),
\end{equation*}
where $\bb_i\in\R,i\in\mathbb{S}$, is given in ({\bf H}), and the
domain
\begin{equation*}
\mathscr{D}(D):=\{f\in L^2(\pi):D(f)<\8\}.
\end{equation*}
The principal eigenvalue $\lambda_0$ of $D(f)$ is defined by
\begin{equation*}
\lambda_0:=\inf\{D(f):f\in\mathscr{D}(D),\|f\|_0=1\}.
\end{equation*}
For more details on the first eigenvalue, refer to \cite[Chapter
3]{Ch}. Due to the fact that the state space of $\{\LL_t\}_{t\ge0}$
is finite, there exists $\xi=(\xi_1,\cdots,\xi_N)\in\mathscr{D}(D)$
such that
\begin{equation}\label{***}
D(\xi)=\lambda_0\|\xi\|^2_0.
\end{equation}
For $\xi\in\mathscr{D}(D)$ such that \eqref{***} holds, set
\begin{equation*}
\tt\xi_1:=\max_{i\in\mathbb{S}}\xi_i,
~~\tt\xi_2:=(\min_{i\in\mathbb{S}}\xi_i)^{-1}.
\end{equation*}
Let
\begin{equation}\label{*0}
 \kappa:=\{(1+12q_0)\bb_0+8L^2(5+6\bb_0)\}\tt\xi_1\tt\xi_2,
\end{equation}
where $q_0,\bb_0$ are given in \eqref{T8}, and $L>0$ defined in
\eqref{eq5}.
%where $Q$ is the $Q$-matrix of $\{\LL_t\}_{t\ge0},$ and $\bb_i\in\R$
%such that ({\bf H}).

\smallskip

The main result in this section is the following.

\begin{thm}\label{eigen}
{\rm Let $N<\8$, \eqref{eq5} and ({\bf H})   hold, and assume
further $\lambda_0>0$. Then, $(\bar Y_{k\dd}^{x,i},\LL_{k\dd}^i)$
admits a unique invariant measure $\pi^\dd\in\mathcal
{P}(\R^n\times\mathbb{S})$ for any
\begin{equation*}
 \dd<(1/(32L^2))\wedge(\lambda_0/\kk)^2.
 \end{equation*}
  }
\end{thm}

\begin{proof}
  Recalling
\eqref{***} and checking  the argument of \cite[Theorem 3.2]{SX},
one has
\begin{equation*}
\xi\gg{\bf0} ~~\mbox{ and
}~~(Q\xi)(i)+\bb_i\xi_i=-\lambda_0\xi_i,~i\in\mathbb{S}.
\end{equation*}
The remainder of the proof is similar to that of
Theorem \ref{existence}, here we omit it.
\end{proof}

Next, an example is constructed to demonstrate  Theorem \ref{eigen}.

\begin{exa}\label{ex4.3}
{\rm Let $\{\LL_t\}_{t\ge0}$ be a right-continuous Markov chain
taking values in $\mathbb{S}:=\{0, 1,2\}$ with the generator
\begin{equation*}
Q= \left(\begin{array}{ccc}
  -b & b &0\\
  2a & -2(a+b) & 2b\\
  0 & 3a &-3a
  \end{array}
  \right)
  \end{equation*}
for some $a,b>0$.  Consider a scalar SDE with regime switching
\begin{equation}\label{**5a}
\d X_t=\aa_{\LL_t}X_t\d t+\si_{\LL_t}X_t\d W_t,~~~t\ge0,~~X_0=x,
\end{equation}
where $\aa_\cdot,\si_\cdot:\mathbb{S}\mapsto\R$ such that
\begin{equation*}
c_0=2\aa_0+\si_0^2<0,~~c_1=2\aa_1+\si_1^2,~~c_2=2\aa_2+\si_2^2.
\end{equation*}
We further assume that
\begin{equation}\label{b}
b+c_0<0,~~a-b-c_1>0,~~a-c_2>0.
\end{equation}
Note that \eqref{eq5} holds with $L=\max_{i\in
\mathbb{S}}\{|\aa_i|+|\si_i|\}$ and ({\bf H}) holds with
\begin{equation*}
\bb_0=c_0,~~\bb_1=c_1,~~,\bb_2=c_2.
\end{equation*}
Set
\begin{equation*}
\OO :=Q+\mbox{diag}(\bb_1,\cdots,\bb_N).
\end{equation*}
By the notion of $\OO$,  for $\xi_i=i+1$, $i=0,1,2,$ we deduce that
\begin{equation*}
\begin{split}
(\OO\xi)(0)=-(-b-c_0)\xi_0, ~~ (\OO\xi)(1)=-(a-b-c_1)\xi_1,~~
(\OO\xi)(2)=-(a-c_2)\xi_2.
\end{split}
\end{equation*}
Taking
\begin{equation*}
\lambda=\min\{-b-c_0,a-b-c_1,a-c_2\}>0
\end{equation*}
thanks to \eqref{b}, one finds that
\begin{equation*}
(\OO\xi)(i)\le-\lambda\xi_i,~~~i=0,1,2.
\end{equation*}
Then $\lambda_0>0$ due to \cite[Theorem 4.4]{SX}. As a result,
 $(\bar
Y_{k\dd}^{x,i},\LL_{k\dd}^i)$ has a unique invariant measure
$\pi^\dd\in\mathcal {P}(\R^n\times\mathbb{S})$ whenever the stepsize
is sufficiently small. }
\end{exa}

\begin{rem}
{\rm  The principal-eigenvalue approach has been applied
successfully to investigate ergodic property, stability and
recurrence for regime-switching diffusion processes. For more
details, please refer to   Shao \cite{S14} and Shao-Xi \cite{SX14}. As
we discuss previously, for the reversible case, such trick can also
be utilized to discuss existence and uniqueness of numerical
invariant measure for RSDPs with multiplicative noises.
% whereas the trick adopted in Theorem
%\ref{boun} can only deal with the additive noise case.
}
\end{rem}

\begin{rem}\label{q1}
{\rm Theorem \ref{eigen} can also be extended into the case of RSDPs
with countable state spaces (i.e. $N=\8$) provided that $\lambda_0$
is attainable, i.e., there exists $f\in L^2(\pi),f\neq0$, such that
$D(f)=\lambda_0\|f\|_0^2.$ For more details, please refer to
\cite[Theorem 3.2]{SX}.

}
\end{rem}

\section{Numerical Invariant Measure: Countable State Space}\label{sec5}
The approach based on the Perron-Frobenius theorem (see Theorem
\ref{boun} and \ref{existence}) is not suitable to the case that
$\mathbb{S}$ is a countable state space, i.e., $N=\8,$ while the
approach based on the principal eigenvalue (see Theorem \ref{eigen})
can be applied to this case under some additional conditions as
being pointed out in Remark \ref{q1}. Now in this section, we shall
introduce another method to deal with the case $N=\8$, based on a
finite partition approach and an $M$-matrix theory.
\begin{defn}
{\rm (see e.g. \cite[Definition 2.9, p.67]{YM06}) A square matrix
$A=(a_{ij})_{n\times n}$ is called a nonsingular $M$-matrix if $A$
can be expressed in the form $A=sI-B$ with $B\gg{\bf0}$ and
$s>\mbox{Ria}(B)$, where $I$ is the $n\times n$ identity matrix and
$\mbox{Ria}(B)$ the spectral radius of $B.$ }

\end{defn}

 We further suppose that
\begin{equation}\label{r3}
K:=\sup_{i\in\mathbb{S}}\bb_i<\8~~~~\mbox{ and }~~~~
\sup_{i\in\mathbb{S}}(-q_{ii})<\8,
\end{equation}
where $\bb_i\in\R$ is given in ({\bf H}). Let us insert $m$ points
in the interval $(-\8,K]$ as follows:
\begin{equation*}
-\8=:k_0<k_1<\cdots<k_m<k_{m+1}:=K.
\end{equation*}
Then, the interval $(-\8,K]$ is divided into $m+1$ sub-intervals
$(k_{i-1},k_i]$  indexed by $i$. Let
\begin{equation*}
F_i:=\{j\in\mathbb{S}:\bb_j\in(k_{i-1},k_i]\},~~i=1,\cdots,m+1.
\end{equation*}
Without loss of generality, we can and do assume that each $F_i$ is
not empty. Then
\begin{equation*}
F:=\{F_1,\cdots,F_{m+1}\}
\end{equation*}
is a finite partition of $\mathbb{S}$. For $i,j=1,\cdots,m+1$, set
\begin{equation*}
q_{ij}^F:=
\begin{cases}
\sup_{r\in F_i}\sum_{k\in
F_j}q_{rk},~~~~~~~~j<i,\\
\inf_{r\in F_i}\sum_{k\in F_j}q_{rk},~~~~~~~~~j>i,\\
-\sum_{j\neq i}q_{ij}^F,~~~~~~~~~~~~~~~~~i=j.
\end{cases}
\end{equation*}
So $Q^F:=(q_{ij}^F)$ is the  $Q$-matrix for some Markov chain with
the state space $\mathbb{S}_0:=\{1,\cdots,m+1\}.$ For
$i=1,\cdots,m+1$, let
\begin{equation*}
\bb_i^F:=\sup_{j\in F_i}\bb_j,~~~~~~~~H_{m+1}:=
\left(\begin{array}{ccccc}
  1 & 1 &1 & \cdots & 1\\
  0 & 1 &1 &   \cdots &1\\
  \vdots & \vdots & \vdots & \cdots & \vdots\\
  0 & 0 & 0 & \cdots & 1\\
  \end{array}
  \right)_{(m+1)\times(m+1)}.
  \end{equation*}

\begin{thm}\label{coun}
{\rm Let $N=\8$, \eqref{eq5}, ({\bf H}) and \eqref{r3}  hold. Assume
further that $\{\LL_t\}_{t\ge0}$  is exponential ergodic and that
\begin{equation*}
-(Q^F+\mbox{diag}(\bb_1^F,\cdots,\bb_{m+1}^F))H_{m+1}
\end{equation*}
is a nonsingular $M$-matrix. Then $(\bar
Y_{k\dd}^{x,i},\LL_{k\dd}^i)$ admits a unique measure
$\pi^\dd\in\mathcal {P}(\R^n\times\mathbb{S})$ whenever the stepsize
$\dd\in(0,1)$ is sufficiently small.

}
\end{thm}

\begin{proof}
Some ideas of the argument go back to \cite[Theorem 4.1]{S14}.
Moreover, we only sketch the argument of Theorem \ref{coun} since it
is analogous to that of Theorem \ref{existence}.

\smallskip

Since $-(Q^F+\mbox{diag}(\bb_1^F,\cdots,\bb_{m+1}^F))H_{m+1}$ is a
nonsingular $M$-matrix, by \cite[Theorem 2.10, p.68]{YM06} there
exists  a vector $\eta^F:=(\eta_1^F,\cdots,\eta_{m+1}^F)^*\gg{\bf0}$
such that
\begin{equation}\label{w3}
(-\lambda_1^F,\cdots,-\lambda_{m+1}^F)^*:=(Q^F+\mbox{diag}(\bb_1^F,\cdots,\bb_{m+1}^F))H_{m+1}\eta^F\ll{\bf0}.
\end{equation}
Set $\xi^F:=H_{m+1}\eta^F$. By the structure of $H_{m+1}$, it is
trivial to see that
\begin{equation*}
\xi_i^F=\eta^F_{m+1}+\cdots+\eta^F_i,~~i=1,\cdots,m+1.
\end{equation*}
This, together with $\eta^F\gg{\bf0}$, yields that $\xi^F\gg{\bf0}$
and $\xi^F_{i+1}<\xi^F_i,i=1,\cdots,m+1$. Next, we extend the vector
$\xi^F$ to be a vector on $\mathbb{S}$ by setting $\xi_r:=\xi_i^F$
for $r\in F_i$.  Moreover, let
$\phi:\mathbb{S}\mapsto\{1,\cdots,m+1\}$ be a map defined by
$\phi(j):=i$ for $j\in F_i.$ Then, by the definition of $\bb_i^F$,
one has
\begin{equation}\label{r1}
\xi_r=\xi_i^F=\xi_{\phi(r)}^F~\mbox{ and }~ \bb_r\le
\bb^F_{\phi(r)}, ~~~r\in F_i.
\end{equation}
 For any $r\in\mathbb{S}$, there exists
$F_i$ such that $r\in F_i$. Recalling the definition of $q_{ij}^F$
and utilizing $\xi^F_{i+1}<\xi^F_i,i=1,\cdots,m+1$, we derive from
\eqref{r1} that, for $r\in F_i$,
\begin{equation}\label{r2}
\begin{split}
(Q\xi)(r) &=\sum_{k<i}\sum_{j\in
F_k}q_{rj}(\xi_j-\xi_r)+\sum_{k>i}\sum_{j\in
F_k}q_{rj}(\xi_j-\xi_r)\\
&=\sum_{k<i}\sum_{j\in
F_k}q_{rj}(\xi_k^F-\xi_i^F)+\sum_{k>i}\sum_{j\in
F_k}q_{rj}(\xi_k^F-\xi_i^F)\\
&\le\sum_{k<i}q_{ik}^F(\xi_k^F-\xi_i^F)+\sum_{k>i}q_{ik}^F(\xi_k^F-\xi_i^F)\\
&=(Q^F\xi^F)(i)=(Q^F\xi^F)(\phi(r)).
\end{split}
\end{equation}
For any $\rr>0$, observe  from \eqref{w3}-\eqref{r2}  that
\begin{equation*}
\begin{split}
&\e^{\rr t}\E(|Y_t^x|^2\xi_{\LL_t^i})\\ &\le c(|x|^2 + \e^{\rr
t})+\E\int_0^t\e^{\rr
s}\{\rr\xi_{\LL_s^i}+(Q\xi)(\LL_s^i)+\bb_{\LL_s^i}\xi_{\LL_s^i}\}|Y_s^x|^2\d
s+\int_0^t\e^{\rr
s}\{\Lambda_1(s)+\Lambda_2(s)\}\d s\\
&\le c(|x|^2 + \e^{\rr t})+\E\int_0^t\e^{\rr
s}\{\rr\xi_{\phi(\LL_s^i)}^F+(Q^F\xi^F)(\phi(\LL_s^i))+
\bb^F_{\phi(\LL_s^i)}\xi_{\phi(\LL_s^i)}^F\}|Y_s^x|^2\d s\\
&\quad+\int_0^t\e^{\rr
s}\{\Lambda_1(s)+\Lambda_2(s)\}\d s\\
&\le  c(|x|^2 + \e^{\rr t})-\E\int_0^t\e^{\rr
s}\{\lambda_{\phi(\LL_s^i)}^F-\rr\xi_{\phi(\LL_s^i)}^F\}|Y_s^x|^2\d
s+\int_0^t\e^{\rr
s}\{\Lambda_1(s)+\Lambda_2(s)\}\d s\\
&\le  c(|x|^2 + \e^{\rr t})-(\lambda_0-\rr\xi_0)\E\int_0^t\e^{\rr
s}|Y_s^x|^2\d s+\int_0^t\e^{\rr s}\{\Lambda_1(s)+\Lambda_2(s)\}\d s
\end{split}
\end{equation*}
where $\lambda_0:=\min_{i\in\mathbb{S}_0}\lambda_i^F>0$ due to
\eqref{w3}, $\xi_0:=\max_{i\in\mathbb{S}_0}\xi_i^F$, and $\Lambda_i$
is defined as in \eqref{T9}. Then, following the argument of Theorem
\ref{existence}, we have
\begin{equation*}
\sup_{t\ge0}\E|Y_t^x|^2<\8.
\end{equation*}
Then, ergodicity of $\{\LL_t\}_{t\ge0}$   yields existence of
numerical invariant measure whenever the stepsize $\dd>0$ is
sufficiently small. The proof of the uniqueness is similar to that of Theorem \ref{coun}, since  Markov chain  $\{\LL_t\}_{t\ge0}$ is exponential ergodic, \eqref{T5} holds.
\end{proof}

\begin{rem}
{\rm $(\bar Y_{k\dd}^{x,i},\LL_{k\dd}^i)$ associated with
\cite[Example 4.1]{S14} admits a numerical
 invariant measure whenever the stepsize $\dd>0$ is
sufficiently small. }
\end{rem}

\end{document}